\documentclass[sn-mathphys,Numbered]{sn-jnl}

\usepackage{graphicx}%
\usepackage{multirow}%
\usepackage{amsmath,amssymb,amsfonts}%
\usepackage{amsthm}%
\usepackage{mathrsfs}%
\usepackage[title]{appendix}%
\usepackage{xcolor}%
\usepackage{textcomp}%
\usepackage{manyfoot}%
\usepackage{booktabs}%
\usepackage{algorithm}%
\usepackage{algorithmicx}%
\usepackage{algpseudocode}%
\usepackage{listings}%

 \newtheorem{thm}{Theorem}[section]
 \newtheorem{cor}[thm]{Corollary}
 
 \newtheorem{prop}[thm]{Proposition}
 \theoremstyle{definition}
 \newtheorem{defn}[thm]{Definition}
 \theoremstyle{remark}
 \newtheorem{rem}[thm]{Remark}
 \newtheorem{ex}[thm]{Example}
 \numberwithin{equation}{section}

\newcommand{\ve}{\varepsilon}

\newcommand{\NN}{\mathbb{N}}

\begin{document}

\title[Generalized Kannan type mappings]{Fixed point theorem for generalized Kannan type mappings}


\author*[1]{\fnm{Evgeniy} \sur{Petrov}}\email{eugeniy.petrov@gmail.com}

\author[2]{\fnm{Ravindra K.} \sur{Bisht}}\email{ravindra.bisht@yahoo.com}

\affil*[1]{
\orgdiv{Function Theory Department},
\orgname{Institute of Applied Mathematics and Mechanics of the NAS of Ukraine}, \orgaddress{\street{Batiuka str. 19},
\city{Slovyansk},
\postcode{84116},
\country{Ukraine}}}

\affil[2]{
\orgdiv{Department of Mathematics},
\orgname{National Defence Academy},
\orgaddress{
\city{Pune},
\state{Khadakwasla},
\country{India}}}



\abstract{We introduce a new type of mappings in metric spaces which are three-point analogue of the well-known Kannan type mappings and call them generalized Kannan type mappings. It is shown that in general case such mappings are discontinuous but continuous at fixed points as well as Kannan type mappings and that these two classes of mappings are independent. The fixed point theorem for generalized Kannan type mappings is proved. Additional conditions of asymptotic regularity and continuity allow us to extent the class of mappings for which the fixed point theorems hold. Following Kannan, we also obtain two other fixed point theorems for generalized Kannan type mappings in metric spaces which are not obligatory complete.}




\maketitle

\section{Introduction}
In~\cite{Ka68} Kannan proved the following result which gives the fixed point for discontinuous mappings: let $T\colon X\to X$ be a mapping on a complete metric space $(X,d)$ with
  \begin{equation}\label{e0}
   d(Tx,Ty)\leqslant \lambda (d(x,Tx)+d(y,Ty))
  \end{equation}
where $0\leqslant \lambda<\frac{1}{2}$ and $x,y \in X$. Then $T$ has a unique fixed point.

If for the mapping $T\colon X\to X$ there exists $0\leqslant \lambda<\frac{1}{2}$ such that the inequality~(\ref{e0}) holds for all $x,y \in X$, then  $T$ is called \emph{Kannan type mapping}.

Subrahmanyam~\cite{S75} proved that this result of Kannan characterizes the metric completeness. That is, a metric space $X$ is complete if and only if every Kannan type mapping on $X$ has a fixed point. Note that Banach contractions do not characterize the metric completeness. For an example of a metric space $X$ such that $X$ is not complete and every contraction on $X$ has a fixed
point see~\cite{Co59}. See, e.g., also papers~\cite{SST98,Su05, KT08,KS08,DLS09,DLS10,Be04,Be07} on the similarities and differences between Banach contractions and Kannan type mappings.

Typically, in the Fixed point theory, it is possible to distinguish at least three types of generalizations of the Kannan theorem:
in the first case, the contractive nature of the mapping is weakened, see, e.g.~\cite{Is74,Re71,Re712,Rh77,Bi72,Go19};
in the second case the topology is weakened, see, e.g.~\cite{Ak22,KI19,AA12,DLG12,Ho22};
the third case  involves theorems for some {K}annan-type multivalued mappings, see, e.g.~\cite{NS10, Um15,DD11}.

In~\cite{P23}, it was introduced a new type of mappings, which can be characterized as mappings contracting perimeters of triangles.
 \begin{defn}\label{d0}
Let $(X,d)$ be a metric space with $|X|\geqslant 3$. We shall say that $T\colon X\to X$ is a \emph{mapping contracting perimeters of triangles} on $X$ if there exists $\alpha\in [0,1)$ such that the inequality
  \begin{equation}\label{mcpt}
   d(Tx,Ty)+d(Ty,Tz)+d(Tx,Tz) \leqslant \alpha (d(x,y)+d(y,z)+d(x,z))
  \end{equation}
  holds for all three pairwise distinct points $x,y,z \in X$.
\end{defn}

\begin{rem}
Note that the requirement for $x,y,z\in X$ to be pairwise distinct is essential. One can see that otherwise this definition is equivalent to the definition of contraction mapping.
\end{rem}
A fixed-point theorem for such mappings was proved. Although the proof of this theorem is based on the ideas of the proof of Banach's classical theorem, the essential difference is that the definition of these mappings is based on the mapping of three points of the space instead of two. Moreover, a condition which prevents these mappings from having periodic points of prime period 2 is required. Notably, ordinary contraction mappings form an important subclass of these mappings.

Following~\cite{P23} in the next definition we introduce three-point analogue of the Kannan type mappings.
\begin{defn}\label{d1}
Let $(X,d)$ be a metric space with $|X|\geqslant 3$. We shall say that $T\colon X\to X$ is a \emph{generalized Kannan type mapping} on $X$ if there exists $\lambda\in [0,\frac{2}{3})$ such that the inequality
  \begin{equation}\label{e1}
   d(Tx,Ty)+d(Ty,Tz)+d(Tx,Tz) \leqslant \lambda (d(x,Tx)+d(y,Ty)+d(z,Tz))
  \end{equation}
  holds for all three pairwise distinct points $x,y,z \in X$.
\end{defn}

In Examples~\ref{exa3} and~\ref{exa1} we show that the classes of Kannan type mappings and generalized Kannan type mappings are independent.

In Section~\ref{sec2}, we study connection between generalized Kannan type mappings, Kannan type mappings and mappings contracting perimeters of triangles. Additionally, we provide an example of a discontinuous generalized Kannan-type mapping.

In Section~\ref{sec3}, we prove the main result of this paper Theorem~\ref{t1}, which is a fixed point theorem for generalized Kannan type mappings. It is noteworthy that this theorem asserts that the number of fixed points is at most two. Furthermore, it is demonstrated that the generalized Kannan-type mappings are continuous at fixed points.

In Section~\ref{sec4}, we study asymptotically regular generalized Kannan type mappings. The condition of asymptotic regularity allows us to extend the value of the parameter $\lambda$ in~(\ref{e1}) to the set $[0,1)$ instead of $[0,\frac{2}{3})$ in the fixed point theorem, see Theorem~\ref{gkt1m}. The additional condition of continuity  allows us to obtain a fixed point theorem for the class of generalized $F$-Kannan type mappings, see Theorem~\ref{gkt2m}, and extend the value of the parameter $\lambda$ to the set $[0,\infty)$, see Corollary~\ref{cgkt2m}.

In Section~\ref{sec5}, following  work~\cite{Ka69} of Kannan, we obtain two other fixed point theorems for generalized Kannan type mappings. In the first case, we omit the requirement for the metric space $X$ to be complete, consider that $T\colon X\to X$ is continuous at some point $x^*\in X$ and that there exists a point $x_0\in X$ such that the sequence of iterates $x_n=Tx_{n-1}$, $n=1,2,...$, has a subsequence $x_{n_k}$, converging to $x^*$, see Theorem~\ref{t2}. In the second case, we suppose additionally that the mapping $T$ is continuous at the whole space but not only at the point $x_0$ and that condition~(\ref{e1}) holds only on an everywhere dense subset of the space, see Theorem~\ref{t3}.

\section{Some properties of generalized Kannan type mappings}\label{sec2}

In this section we study connections between generalized Kannan type mappings, Kannan type mappings and mappings contracting perimeters of triangles.

\begin{prop}
Kannan type mappings with $\lambda\in[0,\frac{1}{3})$ are generalized Kannan type mappings.
\end{prop}
\begin{proof}
Let $(X,d)$ be a metric space with $|X|\geqslant 3$, $T\colon X\to X$ be a Kannan type mapping and let $x,y,z\in X$ be pairwise distinct. Consider inequality~(\ref{e0}) for the pairs
$x, z$ and $y, z$:
  \begin{equation}\label{e02}
   d(Tx,Tz)\leqslant \lambda (d(x,Tx)+d(z,Tz)),
  \end{equation}
  \begin{equation}\label{e01}
   d(Ty,Tz)\leqslant \lambda (d(y,Ty)+d(z,Tz)).
  \end{equation}
Summarizing the left and the right parts of inequalities~(\ref{e0}), ~(\ref{e02}) and~(\ref{e01}) we obtain
  \begin{equation*}
   d(Tx,Ty)+d(Ty,Tz)+d(Tx,Tz) \leqslant 2\lambda (d(x,Tx)+d(y,Ty)+d(z,Tz))
  \end{equation*}
Hence, we get the desired assertion.
\end{proof}

\begin{prop}\label{p1}
Let $(X,d)$ be a metric space and let  $T\colon X\to X$ be a generalized Kannan type mapping with some  $\lambda\in [0,\frac{2}{3})$. If $x$ is an accumulation point of  $X$ and $T$ is continuous at $x$, then the inequality
\begin{equation}\label{w1}
 d(Tx,Ty)\leqslant \lambda\left(d(x,Tx)+\frac{d(y,Ty)}{2}\right)
\end{equation}
holds for all points $y\in X$.
\end{prop}

\begin{proof}
Let $x\in X$ be an accumulation point and let $y\in X$. If $y=x$, then clearly~(\ref{w1}) holds. Let now $y\neq x$. Since $x$ is an accumulation point, then there exists a sequence $z_n\to x$ such that $z_n\neq x$, $z_n\neq y$ and all $z_n$ are different.
Hence, by~(\ref{e1}) the inequality
  \begin{equation*}
   d(Tx,Ty)+d(Ty,Tz_n)+d(Tx,Tz_n) \leqslant \lambda (d(x,Tx)+d(y,Ty)+d(z_n,Tz_n))
  \end{equation*}
holds for all $n\in \NN$. Since $z_n\to x$ and  $T$ is continuous at $x$ we have $Tz_n \to Tx$. Since every metric is continuous we have $d(z_n, Tz_n) \to d(x,Tx)$. Hence, we get~(\ref{w1}).
\end{proof}
\begin{cor}\label{cor1}
Let $(X,d)$ be a metric space, $T\colon X\to X$ be a continuous generalized Kannan type mapping and let all points of $X$ are accumulation points. Then $T$ is a Kannan type mapping.
\end{cor}
\begin{proof}
According to Proposition~\ref{p1}, inequality~(\ref{w1}) holds as well as the inequality
\begin{equation}\label{w2}
 d(Tx,Ty)\leqslant \lambda\left(d(y,Ty)+\frac{d(x,Tx)}{2}\right).
\end{equation}
Summarizing the left and the right parts of ~(\ref{w1}) and ~(\ref{w2}) and dividing both parts of the obtained inequality by $2$ we get
\begin{equation*}
 d(Tx,Ty)\leqslant \frac{3\lambda}{4}\left(d(x,Tx)+d(y,Ty)\right).
\end{equation*}
Since $\lambda\in [0,\frac{2}{3})$, we have $\frac{3\lambda}{4} \in [0,\frac{1}{2})$, which completes the proof.
\end{proof}

\begin{prop}
Let $(X, d)$ be a metric space with $|X| \geq 3$, and $T\colon X \to X$ be a mapping contracting perimeters of triangles with a constant $0\leqslant \alpha < \frac{1}{4}$. Then $T$ is a generalized Kannan type mapping with respect to the metric $d$.
\end{prop}
\begin{proof}
Applying several times triangle inequality to the right part of~(\ref{mcpt}) we get
\begin{multline*}\label{e12}
    d(Tx,Ty)+d(Ty,Tz)+d(Tx,Tz)
    \\ \leqslant  \alpha  (d(x,Tx)+d(Tx,y)+d(y,Ty)+d(Ty,z)+d(z,Tz)+d(Tz,x))
    \\ \leqslant  \alpha  (d(x,Tx)+d(Tx,Ty)+d(Ty,y)+d(y,Ty)
    \\+d(Ty,Tz)+d(Tz,z)+d(z,Tz)+d(Tz,Tx)+d(Tx,x)).
\end{multline*}
Rearranging this inequality, we obtain

\begin{equation*}
 d(Tx, Ty)+d(Ty,Tz)+d(Tz,Tx) \leq \frac{2\alpha}{1-\alpha} (d(x, Tx) + d(y,Ty) + d(z,Tz))
\end{equation*}
for  all three pairwise distinct points $x, y, z\in X$.

Since $0 < \lambda = \frac{2\alpha}{1 - \alpha} < \frac{2}{3}$, we have that $T$ is a generalized Kannan type mapping.
\end{proof}

\begin{ex}\label{exa3}
Let us construct an example of the Kannan-type mapping which is not  a generalized Kannan-type mapping. Let $X=[0,1]$ and let $d$ be the Euclidean distance on $X$. Consider a mapping $T\colon X\to X$ such that $T(x)=x/a$ for some real $a>1$. Without loss of generality consider that $x\geqslant y$. Consider inequality~(\ref{e0}) for this mapping
$$
\frac{x}{a}-\frac{y}{a}\leqslant \lambda\left(x-\frac{x}{a} +y - \frac{y}{a}\right).
$$
Hence,
\begin{equation}\label{ex3}
x-y\leqslant \lambda(a-1)(x+y).
\end{equation}
It is easy to see that inequality~(\ref{ex3}) holds for all $x\geqslant y$ iff $\lambda(a-1)\geqslant 1$.
Consider the system of inequalities
$$
\begin{cases}
0\leqslant \lambda <\frac{1}{2},\\
\lambda(a-1)\geqslant 1.
\end{cases}
$$
Hence,
$$
\frac{1}{a-1}\leqslant \lambda <\frac{1}{2}.
$$
Thus, $T$ is a Kannan type mapping if $a>3$.

Further, without loss of generality suppose that $x>y>z$. Consider now inequality~(\ref{e1}) for the same mapping $T$ and pairwise distinct $x,y,z \in [0,1]$:
$$
\frac{1}{a}(x-y+x-z+y-z)\leqslant \lambda\left(1-\frac{1}{a}\right)\left(x+y+z\right),
$$
$$
(2x-2z)\leqslant \lambda\left(a-1\right)\left(x+y+z\right),
$$

\begin{equation}\label{ex31}
x-z\leqslant \frac{\lambda}{2}\left(a-1\right)\left(x+y+z\right).
\end{equation}
It is easy to see that inequality~(\ref{ex31}) holds for all $x>y> z$ iff $\frac{\lambda}{2}(a-1)\geqslant 1$.
Consider the system of inequalities
$$
\begin{cases}
0\leqslant \lambda <\frac{2}{3},\\
\lambda(a-1)\geqslant 2.
\end{cases}
$$
Hence,
\begin{equation}\label{ex48}
\frac{2}{a-1}\leqslant \lambda <\frac{2}{3}.
\end{equation}
Thus, $T$ is a generalized Kannan type mapping only if $a>4$.
Hence for $a\in (3,4]$ the mapping $T$ is a Kannan type mapping but not a generalized Kannan type mapping.
\end{ex}

\begin{ex}
Let us show that generalized Kannan type mappings in general case are discontinuous as well as Kannan type mappings.
Let $X=[0,1]$ and let $d$ be the Euclidean distance on $X$. Consider a mapping $T\colon X\to X$ such that
$$
T(x)=
\begin{cases}
\frac{x}{a},\quad x\in [0,\frac{1}{2}],\\
\frac{x}{b},\quad x\in (\frac{1}{2},1],\\
\end{cases}
$$
where $a,b>1$, $a\neq b$. It is clear that $T$ is discontinuous at $x=\frac{1}{2}$.
Let us show that there exist $a$ and $b$ such that inequality~(\ref{e1}) holds for all three pairwise distinct points $x, y, z \in X$ and for some $0\leqslant \lambda<\frac{2}{3}$, i.e. that $T$ is, indeed, a generalized Kannan type mapping. Without loss of generality consider that  $x>y>z$. It is cleat that it is enough to consider the following cases:

1) $x,y,z\in [0,\frac{1}{2}]$,

2) $x,y,z\in (\frac{1}{2},1]$,

3) $z, y \in [0,\frac{1}{2}]$, $x\in (\frac{1}{2},1]$,

4) $z \in [0,\frac{1}{2}]$, $y,x\in (\frac{1}{2},1]$.

\noindent Using the previous example, we see that cases 1) and 2) give us restrictions~(\ref{ex48}) and
\begin{equation}\label{ex49}
\frac{2}{b-1}\leqslant \lambda <\frac{2}{3}.
\end{equation}

Further, without loss of generality consider that $a>b$. Consider inequality~(\ref{e1}) for case 3)
$$
\frac{x}b-\frac{y}a+\frac{x}b-\frac{z}a+\frac{y}a-\frac{z}a\leqslant \lambda\left(x-\frac{x}b+y-\frac{y}a+z-\frac{z}a\right).
$$
Hence,

\begin{equation}\label{ex45}
0\leqslant x\left(\lambda-\frac{\lambda}{b}-\frac{2}{b}\right)
+y\left(\lambda-\frac{\lambda}{a}\right)
+z\left(\lambda-\frac{\lambda}{a}+\frac{2}{a}\right).
\end{equation}

\noindent Consider case 4)
$$
\frac{x}b-\frac{y}b+\frac{x}b-\frac{z}a+\frac{y}b-\frac{z}a\leqslant \lambda\left(x-\frac{x}b+y-\frac{y}b+z-\frac{z}a\right).
$$
Hence,
\begin{equation}\label{ex46}
0\leqslant x\left(\lambda-\frac{\lambda}{b}-\frac{2}{b}\right)
+y\left(\lambda-\frac{\lambda}{b}\right)
+z\left(\lambda-\frac{\lambda}{a}+\frac{2}{a}\right).
\end{equation}
It is clear that for every $0<\lambda <\frac{2}{3}$ there exist sufficiently large $a$ and $b$, $a>b$, such that the inequalities~(\ref{ex48}), (\ref{ex49}), (\ref{ex45}) and~(\ref{ex46}) hold simultaneously. Hence, $T$ is a discontinuous generalized Kannan type mapping.
\end{ex}

\section{The main result}\label{sec3}

Let $T$ be a mapping on the metric space $X$. A point $x\in X$ is called a \emph{periodic point of period $n$} if $T^n(x) = x$. The least positive integer $n$ for which $T^n(x) = x$ is called the prime period of $x$, see, e.g.,~\cite[p.~18]{De22}. In particular, the point $x$ is of prime period $2$ if $T(T(x))=x$ and $Tx\neq x$.

\begin{rem}
Generalized Kannan type mappings can not possess periodic points of prime period three. Indeed, let for some point $x$ we have $Tx=y$,  $y\neq x$, $Ty=z$, $x\neq z\neq y$, and $Tz=x$. Then we have the equality
$$
 d(Tx,Ty)+d(Ty,Tz)+d(Tx,Tz) = d(x,Tx)+d(y,Ty)+d(z,Tz),
$$
which contradicts to~(\ref{e1}).
\end{rem}

The following theorem is the main result of this paper.
\begin{thm}\label{t1}
Let $(X,d)$, $|X|\geqslant 3$, be a complete metric space and let the mapping $T\colon X\to X$ satisfy the following two conditions:
\begin{itemize}
  \item [(i)] $T$ does not possess periodic points of prime period $2$.
  \item [(ii)] $T$ is a generalized Kannan-type mapping on $X$.
\end{itemize}
Then $T$ has a fixed point. The number of fixed points is at most two.
\end{thm}

\begin{proof}
Let $x_0\in X$, $Tx_0=x_1$, $Tx_1=x_2$, \ldots, $Tx_n=x_{n+1}$, \ldots. Suppose that $x_n$ is not a fixed point of the mapping $T$ for every $n=0,1,...$.
Since $x_{n-1}$ is not fixed, then $x_{n-1}\neq x_n=Tx_{n-1}$. By condition (i) $x_{n+1}=T(T(x_{n-1}))\neq x_{n-1}$ and by the supposition that $x_{n}$ is not fixed we have $x_n\neq x_{n+1}=Tx_n$. Hence, $x_{n-1}$, $x_n$ and $x_{n+1}$ are pairwise distinct.
Let us set in~(\ref{e1}) $x=x_{n-1}$, $y=x_n$, $z=x_{n+1}$.
Then
$$
d(Tx_{n-1},Tx_n)+d(Tx_n,T_{x_{n+1}})+d(Tx_{n-1}, Tx_{n+1})
$$
$$
\leqslant
\lambda(d(x_{n-1},Tx_{n-1})+d(x_n,Tx_n)+d(x_{n+1}, Tx_{n+1}))
$$
and
$$
d(x_{n},x_{n+1})+d(x_{n+1},x_{n+2})+d(x_{n+2}, x_{n})
$$
$$
\leqslant
\lambda(d(x_{n-1},x_{n})+d(x_n,x_{n+1})+d(x_{n+1}, x_{n+2})).
$$
Hence,
$$
(1-\lambda)d(x_{n+1},x_{n+2})\leqslant \lambda(d(x_{n-1},x_n)+d(x_n,x_{n+1}))-d(x_n,x_{n+1})-d(x_{n+2},x_n).
$$
Using the triangle inequality $d(x_{n+1}, x_{n+2})\leqslant d(x_n, x_{n+1})+d(x_{n+2},x_n)$ we get
$$
(1-\lambda)d(x_{n+1},x_{n+2})
\leqslant \lambda (d(x_{n-1},x_n)+d(x_n,x_{n+1}))-d(x_{n+1},x_{n+2}).
$$
Further,
$$
(2-\lambda)d(x_{n+1},x_{n+2})\leqslant \lambda (d(x_{n-1},x_n)+d(x_n,x_{n+1})),
$$
$$
d(x_{n+1},x_{n+2})\leqslant \frac{\lambda}{2-\lambda} (d(x_{n-1}, x_{n})+d(x_{n}, x_{n+1}))
$$
and
$$
d(x_{n+1},x_{n+2})\leqslant \frac{2\lambda}{2-\lambda} \max\{d(x_{n-1}, x_{n}),d(x_{n}, x_{n+1})\}.
$$
Let $\alpha = \frac{2\lambda}{2-\lambda}$. Using the relation $\lambda \in [0,\frac{2}{3})$, we get $\alpha \in [0,1)$. Further,
\begin{equation}\label{e2}
d(x_{n+1},x_{n+2})\leqslant \alpha \max\{d(x_{n-1}, x_{n}),d(x_{n}, x_{n+1})\}.
\end{equation}

Set $a_n=d(x_{n-1},x_n)$, $n=1,2,\ldots,$ and let
$a=\max\{a_{1},a_{2}\}$.
Hence and by~(\ref{e2}) we obtain
$$
a_1\leqslant a, \, \,
a_2\leqslant a, \, \,
a_3\leqslant \alpha a, \, \,
a_4\leqslant \alpha a, \, \,
a_5\leqslant \alpha^2 a, \, \,
a_6\leqslant \alpha^2 a, \,\,
a_7\leqslant \alpha^3 a, \,
\ldots .
$$
Since $\alpha <1$, it is clear that the inequalities
$$
a_1\leqslant a, \, \,
a_2\leqslant a, \, \,
a_3\leqslant \alpha^{\frac{1}{2}} a, \, \,
a_4\leqslant \alpha a, \, \,
a_5\leqslant \alpha^{\frac{3}{2}} a, \, \,
a_6\leqslant \alpha^2 a, \,\,
a_7\leqslant \alpha^{\frac{5}{2}} a, \,
\ldots
$$
also hold. I.e.,
\begin{equation}\label{e9}
a_n\leqslant \alpha^{\frac{n}{2}-1}a
\end{equation}
for $n=3,4,\ldots$.

Let $p\in \mathbb N$, $p\geqslant 2$. By the triangle inequality, for $n\geqslant 3$ we have
$$
d(x_n,\,x_{n+p})\leqslant d(x_{n},\,x_{n+1})+d(x_{n+1},\,x_{n+2})+\ldots+d(x_{n+p-1},\,x_{n+p})
$$
$$
=a_{n+1}+a_{n+2}+\cdots+a_{n+p} \leqslant
a(\alpha^{\frac{n+1}{2}-1}+\alpha^{\frac{n+2}{2}-1}+\cdots
+\alpha^{\frac{n+p}{2}-1})
$$
$$
=a\alpha^{\frac{n+1}{2}-1}(1+\alpha^{\frac{1}{2}}+\cdots
+\alpha^{\frac{p-1}{2}})=a\alpha^{\frac{n-1}{2}}\frac{1-\sqrt{\alpha^p}}{1-\sqrt{\alpha}}.
$$
Since by the supposition $0\leqslant\alpha<1$, then $0\leqslant \sqrt{\alpha^p}<1$ and $d(x_n,\,x_{n+p})\leqslant a\alpha^{\frac{n-1}{2}}\frac{1}{1-\sqrt{\alpha}}$. Hence, $d(x_n,\,x_{n+p})\to 0$ as $n\to \infty$ for every $p>0$. Thus, $\{x_n\}$ is a Cauchy sequence. By the completeness of $(X,d)$, this sequence has a limit $x^*\in X$.

Recall that any three consecutive element of the sequence $(x_n)$ are pairwise distinct. If $x^*\neq x_k$ for all $k\in \{1,2,...\}$, then inequality~(\ref{e1}) holds for the pairwise distinct points  $x^*$, $x_{n-1}$ and $x_n$.
Suppose that there exists the smallest possible $k\in \{1,2,...\}$ such that $x^*=x_k$.  Let $m>k$ be such that $x^*=x_m$, then the sequence $(x_n)$ is cyclic starting from $k$ and can not be a Cauchy sequence. Hence, the points $x^*$, $x_{n-1}$ and $x_n$ are pairwise distinct at least when $n-1>k$.

Let us prove that $Tx^*=x^*$. If there exists $k\in \{1,2,...\}$ such that $x_k=x^*$, then suppose that $n-1>k$. By the triangle inequality and by inequality~(\ref{e1}) we have
$$
d(x^*,Tx^*)\leqslant d(x^*,x_{n})+d(x_{n},Tx^*)
=d(x^*,x_{n})+d(Tx_{n-1},Tx^*)
$$
$$
\leqslant d(x^*,x_{n})+d(Tx_{n-1},Tx^*)+d(Tx_{n-1},Tx_{n})+d(Tx_{n},Tx^*)
$$
$$
\leqslant d(x^*,x_{n})+\lambda(d(x_{n-1},Tx_{n-1})
+d(x_{n},Tx_{n})+d(x^*,Tx^*)).
$$
Hence,
$$
d(x^*,Tx^*)(1-\lambda)\leqslant d(x^*,x_{n})+\lambda(d(x_{n-1},x_{n})
+d(x_{n},x_{n+1}))
$$
and

\begin{equation}\label{e50}
d(x^*,Tx^*)\leqslant
\frac{1}{1-\lambda}(d(x^*,x_{n})+\lambda(d(x_{n-1},x_{n})
+d(x_{n},x_{n+1}))).
\end{equation}

Since all the distances in the right side tend to zero as $n\to \infty$, we obtain $d(x^*,Tx^*)=0$.

Suppose that there exists at least three pairwise distinct fixed points $x$, $y$ and $z$.  Then $Tx=x$, $Ty=y$ and $Tz=z$, which contradicts to~(\ref{e1}).
\end{proof}

\begin{rem}\label{rem1}
Theorem~\ref{t1} does not hold for $\lambda> \frac{2}{3}$. Indeed, suppose that $\lambda > \frac{2}{3}$.
Let $X=\{x,y,z,t\}$,
$d(x,y)=d(y,z)=d(z,t)=d(t,x)=a$,
$d(x,z)=d(y,t)=\varepsilon$. It is easy to see that $(X,d)$ is a metric space for all $0<\ve\leqslant 2a$.
Define $T\colon X\to X$ as
$Tx=y$,
$Ty=z$,
$Tz=t$ and
$Tt=x$.
Consider inequality~(\ref{e1}) for the triplet of points $\{x,y,z\}$:
$$
a+a+\ve\leqslant \lambda (a+a+a).
$$
Hence,
\begin{equation}\label{el}
\frac{2a+\ve}{3a}\leqslant \lambda.
\end{equation}
Clearly, for every $\lambda> \frac{2}{3}$ there exists sufficiently small $\ve$ such that inequality~(\ref{el}) holds.
Note that considering inequality~(\ref{e1}) for the triplets of points $\{y,z,t\}$ $\{z,t,x\}$ $\{t,x,y\}$ we obtain the same inequality~(\ref{el}). Thus, for every $\lambda> \frac{2}{3}$ there exists a complete metric space $X$ with $T\colon X\to X$ such that: 1) inequality~(\ref{e1}) holds with the coefficient $\lambda$ for all pairwise distinct triplets of points of the space $X$; 2) $T$ does not possess periodic points of prime period two; 3) $T$ does not have any fixed point.
\end{rem}

\noindent\textbf{Open problem.} Does Theorem~\ref{t1} hold for $\lambda=\frac{2}{3}$?

\begin{rem}
Suppose that under the supposition of the theorem the mapping $T$ has a fixed point $x^*$ which is a limit of some iteration sequence $x_0, x_1=Tx_0, x_2=Tx_1,\ldots$ such that $x_n\neq x^*$ for all $n=1,2,\ldots$. Then $x^*$ is a unique fixed point.
Indeed, suppose that $T$ has another fixed point $x^{**}\neq x^*$.
It is clear that $x_n\neq x^{**}$ for all $n=1,2,\ldots$. Hence, we have that the points $x^*$, $x^{**}$ and $x_n$ are pairwise distinct for all $n=1,2,\ldots$. Consider the ratio
$$
R_n=\frac{d(Tx^*,Tx^{**})+d(Tx^*,Tx_{n})+d(Tx^{**},Tx_{n})}
{d(x^*,Tx^{*})+d(x_{n},Tx_{n})+d(x^{**},Tx^{**})}
$$
$$
=\frac{d(x^*,x^{**})+d(x^*,x_{n+1})+d(x^{**},x_{n+1})}
{d(x_n,x_{n+1})}.
$$
Taking into consideration that $d(x^*,x_{n+1})\to 0$, $d(x^{**},x_{n+1})\to d(x^{**},x^*)$ and $d(x_n,x_{n+1})\to 0$, we obtain
$R_n\to \infty$ as $n\to \infty$, which contradicts to condition~(\ref{e1}).
\end{rem}

\begin{ex}\label{exa1}
Let us construct an example of generalized Kannan-type mapping $T$, which has exactly two fixed points.
Let $X=\{x,y,z\}$,
$d(x,y)=1$,
$d(y,z)=4$,
$d(x,z)=4$, and let $T\colon X\to X$ be such that $Tx= x$, $Ty= y$ and $Tz= x$. One can easily see that condition (i) of Theorem~\ref{t1} is fulfilled and inequality~(\ref{e1}) holds with $\lambda= \frac{1}{2}$. Note also that $T$ is not a Kannan type mapping since inequality~(\ref{e0}) does not hold for any $0\leqslant \lambda <\frac{1}{2}$.
\end{ex}

\begin{ex}\label{exa2}
Let us show that condition (i) of Theorem~\ref{t1} is necessary. Let the space $(X,d)$ be like in the previous example and let $T\colon X\to X$ be such that $Tx=y$, $Ty=x$ and $Tz=x$. One can easily see that inequality~(\ref{e1}) is fulfilled with any $\frac{1}{3}\leqslant \lambda <\frac{2}{3} $ but $T$ does not have any fixed point.
\end{ex}

It is well-known that Kannan type mappings are continuous at fixed points~\cite{Rh88}. The following proposition shows that generalized Kannan type mappings also have this property.

\begin{prop}
Generalized Kannan type mapping are continuous at fixed points.
\end{prop}

\begin{proof}
Let $(X,d)$ be a metric space with $|X|\geqslant 3$, $T\colon X\to X$ be a \emph{generalized Kannan type mapping} and $x^*$ be a fixed point of $T$. Let $(x_n)$ be a sequence such that $x_n\to x^*$, $x_{n}\neq x_{n+1}$ and $x_{n}\neq x^*$ for all $n$. Let us show that $Tx_n\to Tx^*$.
By~(\ref{e1}) we have
  \begin{multline*}
   d(Tx^*,Tx_n)+d(Tx_n,Tx_{n+1})+d(Tx_{n+1},Tx^*) \\ \leqslant \lambda (d(x^*,Tx^*)+d(x_n,Tx_n)+d(x_{n+1},Tx_{n+1})).
  \end{multline*}
Hence,
  \begin{equation*}
   d(Tx^*,Tx_n)+d(Tx_{n+1},Tx^*) \leqslant \lambda (d(x_n,Tx_n)+d(x_{n+1},Tx_{n+1})).
  \end{equation*}
By the triangle inequality we have
  \begin{multline*}
   d(Tx^*,Tx_n)+d(Tx_{n+1},Tx^*) \\ \leqslant \lambda (d(x_n,x^*)+d(x^*,Tx_n)+d(x_{n+1},x^*)+d(x^*,Tx_{n+1})).
  \end{multline*}
Further,
 $$
   d(Tx^*,Tx_n)+d(Tx_{n+1},Tx^*) \leqslant \frac{\lambda}{1-\lambda}(d(x_n,x^*)+d(x_{n+1},x^*)).
$$
Since $d(x_n,x^*)\to 0$ and $d(x_{n+1},x^*)\to 0$ we have $$d(Tx^*,Tx_n)+d(Tx_{n+1},Tx^*) \to 0$$ and, hence, $d(Tx^*,Tx_n) \to 0$.

Let now $(x_n)$ be a sequence such that $x_n\to x^*$, and
$x_{n}\neq x^*$ for all $n$, but $x_{n} = x_{n+1}$ is possible. Let $(x_{n_k})$ be a subsequence of $(x_n)$ obtained by deleting corresponding repeating elements of $(x_n)$, i.e.,  such that $x_{n_k}\neq x_{n_{k+1}}$ for all $k$. It is clear that $x_{n_k}\to x^*$. As was just proved $Tx_{n_k}\to Tx^*=x^*$. The difference between $Tx_{n_k}$ and $Tx_{n}$ is that $Tx_{n}$ can be obtained from $Tx_{n_k}$ by inserting corresponding repeating consecutive elements. Hence, it is easy to see that $Tx_{n}\to Tx^*$.

Let $(x_n)$ be a sequence such that $x_n=x^*$ for all $n>N$, where $N$ is some natural number. Then, clearly, $Tx_{n}\to Tx^*$.
Let $(x_n)$ now be an arbitrary sequence such that $x_n\to x^*$ but not like in the previous case. Consider a subsequence  $(x_{n_k})$ obtained from $(x_n)$ by deleting elements $x^*$ (if they exist). Clearly, $x_{n_k}\to x^*$. It was just shown that such $Tx_{n_k}\to Tx^*$. Again, we see that $Tx_n$ can be obtained from $Tx_{n_k}$ by inserting in some places elements $Tx^*=x^*$. Again, it is easy to see that $Tx_{n}\to Tx^*$.
\end{proof}

\section{Asymptotic regularity}\label{sec4}

The concept of asymptotic regularity allows us to extend the class of mappings for which the fixed point theorems hold.
\begin{defn}
Let $(X, d)$ be a metric space. A mapping $T\colon X \to X$ satisfying the condition
\begin{equation}\label{ar}
\lim_{n \to \infty} d(T^{n+1}(x), T^n(x)) = 0
\end{equation}
for all $x \in X$ is called asymptotically regular \cite{BP66}.
\end{defn}

\begin{rem}\label{r32}
Let $(X, d)$ be a metric space $T\colon X \to X$ be a self-mapping and let $x_0\in X$,  $T(x_0) = x_1$, $T(x_1) = x_2$, and so on.
If $T$ is asymptotically regular and the sequence $(x_n)$ does not possess a fixed point of $T$, then all the points $x_i$, $i\geqslant 0$, are pairwise distinct. Indeed, otherwise the sequence $(x_n)$ is cyclic starting from some number and condition~(\ref{ar}) does not hold.
\end{rem}

\begin{thm} \label{gkt1m}
Let $(X, d)$ be a complete metric space with $|X| \geqslant 3$ and let the mapping
$T\colon X \to X$ be asymptotically regular generalized Kannan type mapping with the coefficient $\lambda \in [0,1)$. Then $T$ has a fixed point. The number of fixed points is at most two.
\end{thm}

\begin{proof}
\textbf{(i)} Let $x_0 \in X$, $T(x_0) = x_1$, $T(x_1) = x_2$, and so on. Suppose that $(x_n)$ does not possess a fixed point of $T$. Let us prove that $(x_n)$ is a Cauchy sequence. It is sufficient to show that $d(x_n,x_{n+p})\to 0$ as $n\to \infty$ for all $p>0$.
If $p=1$, then this follows from the definition of asymptotic regularity. Let $p\geqslant 2$. By Remark~\ref{r32} the points $x_n$, $x_{n+p-1}$, $x_{n+p}$ are pairwise distinct. Using repeated triangle inequality, inequality~(\ref{e1}) and asymptotic regularity, we get
\begin{align*}
& {} d(x_n,x_{n+p})\leqslant d(x_n,x_{n+1})+d(x_{n+1},x_{n+p+1})+d(x_{n+p+1},x_{n+p})\\\leqslant & {}
d(x_n,x_{n+1})+d(x_{n+1},x_{n+p+1})+d(x_{n+p+1},x_{n+p})+d(x_{n+1},x_{n+p})\\\leqslant & {}
d(x_n,x_{n+1})+\lambda(d(x_{n},Tx_{n})+d(x_{n+p},Tx_{n+p})+d(x_{n+p-1},Tx_{n+p-1}))\\= & {}
d(x_n,x_{n+1})+\lambda(d(x_{n},x_{n+1})+d(x_{n+p},x_{n+p+1})+d(x_{n+p-1},x_{n+p})) \to 0
\end{align*}
as $n\to \infty$. Thus, $(x_{n})$ is a Cauchy sequence. The rest of the proof follows from the proof given in Theorem~\ref{t1}, see~(\ref{e50}).
\end{proof}

Below we show that the assumption of continuity for the mappings $T$ allows us to obtain fixed point theorems for more general classes of mappings than generalized Kannan type mappings even with the coefficient $\lambda\in [0,1)$.

We now provide a more general version of generalized Kannan type mappings and introduce the following definitions. First, we define the class $\mathcal{F}$ of functions $F\colon \mathbb{R}^+ \times \mathbb{R}^+ \times \mathbb{R}^+ \rightarrow \mathbb{R}^+$ satisfying the following conditions:

\begin{enumerate}
    \item[(i)] $F(0, 0, 0) = 0$;
    \item[(ii)] $F$ is continuous at $(0, 0, 0)$.
\end{enumerate}

\begin{defn}\label{d3}
Let $(X,d)$ be a metric space with $|X|\geqslant 3$. We shall say that $T\colon X\to X$ is a \emph{generalized $F$-Kannan type mapping} on $X$ if there exists $F \in \mathcal{F}$ such that the inequality
  \begin{equation}\label{e3}
   d(Tx,Ty)+d(Ty,Tz)+d(Tx,Tz) \leqslant F(d(x,Tx),d(y,Ty), d(z,Tz))
  \end{equation}
  holds for all three pairwise distinct points $x,y,z \in X$.
\end{defn}

\begin{thm} \label{gkt2m}
Let $(X, d)$ be a complete metric space with $|X| \geqslant 3$ and let the mapping
$T\colon X \to X$ be a continuous, asymptotically regular generalized $F$-Kannan type mapping. Then $T$ has a fixed point. The number of fixed points is at most two.
\end{thm}

\begin{proof}
Let $x_0 \in X$, $T(x_0) = x_1$, $T(x_1) = x_2$, and so on. Suppose that $(x_n)$ does not possess a fixed point of $T$. Let us prove that $(x_n)$ is a Cauchy sequence. It is sufficient to show that $d(x_n,x_{n+p})\to 0$ as $n\to \infty$ for all $p>0$.
If $p=1$, then this follows from the definition of asymptotic regularity. Let $p\geqslant 2$. By Remark~\ref{r32} the points $x_n$, $x_{n+p-1}$, $x_{n+p}$ are pairwise distinct. Using repeated triangle inequality, inequality~(\ref{e3}) and asymptotic regularity, we get
\begin{align*}
& {} d(x_n,x_{n+p})\leqslant d(x_n,x_{n+1})+d(x_{n+1},x_{n+p+1})+d(x_{n+p+1},x_{n+p})\\\leqslant & {}
d(x_n,x_{n+1})+d(x_{n+1},x_{n+p+1})+d(x_{n+p+1},x_{n+p})+d(x_{n+1},x_{n+p})\\\leqslant & {}
d(x_n,x_{n+1})+F(d(x_{n},Tx_{n}),d(x_{n+p},Tx_{n+p}), d(x_{n+p-1},Tx_{n+p-1}))\\= & {}
d(x_n,x_{n+1})+ F(d(x_{n},x_{n+1}),d(x_{n+p},x_{n+p+1}), d(x_{n+p-1},x_{n+p})) \to 0
\end{align*}
as $n\to \infty$. Thus, $(x_{n})$ is a Cauchy sequence.  By the completeness of $(X, d)$, this sequence has a limit $x^* \in X$. Now
$$
d(Tx^*,x^*)\leqslant d(Tx^*,x_n)+d(x_n,x^*)= d(Tx^*,Tx_{n-1})+d(x_n,x^*)
$$
Since $T$ is continuous letting $n\to \infty$, we obtain $Tx^*=x^*$.
The rest of the proof follows from reasoning similar to the last paragraph of the proof of Theorem~\ref{t1}.
\end{proof}

Let $\mathcal{B}$ represent the class of functions $\beta: [0,\infty) \to [0,\infty)$ that meet the following condition: $\limsup\limits_{t\rightarrow 0}\beta(t) <\infty$.

\begin{defn}\label{d4}
Let $(X,d)$ be a metric space with $|X|\geqslant 3$. We shall say that $T\colon X\to X$ is a \emph{generalized $\mathcal{B}$-Kannan type mapping} on $X$ if there exist $\beta_1, \beta_2, \beta_3 \in \mathcal{B}$ such that the inequality
\begin{equation}\label{e4}
   \begin{aligned}
      &d(Tx,Ty) + d(Ty,Tz) + d(Tx,Tz) \\
      &\leq \beta_1(d(x,Tx))d(x,Tx) + \beta_2(d(y,Ty))d(y,Ty) + \beta_3(d(z,Tz))d(z,Tz)
   \end{aligned}
\end{equation}
holds for all three pairwise distinct points $x,y,z \in X$.
\end{defn}

\begin{cor}
Let $(X, d)$ be a complete metric space with $|X| \geqslant 3$ and let the mapping
$T\colon X \to X$ be a continuous, asymptotically regular generalized $\mathcal{B}$-Kannan type mapping. Then $T$ has a fixed point. Moreover, the number of fixed points is at most two.
\end{cor}
\begin{proof}
Set $F(x,y,z)=\beta_1(x)x+\beta_2(x)x+\beta_3(x)$. Then  $F(0,0,0)=0$ and $\lim\limits_{x, y, z \to 0} F(x,y,z)=0$ since  $\limsup\limits_{t\rightarrow 0} \beta_i(t) <\infty$ for $i=1,2,3$. Hence, this assertion follows from Theorem~\ref{gkt2m}.
\end{proof}


By setting $\beta_1(t)=\beta_2(t)=\beta_3(t)= \lambda \geqslant 0$ in (\ref{e4}), we get a generalized Kannan type mapping with the coefficient $\lambda \in [0,\infty)$. Hence, we immediately obtain the following.

\begin{cor}\label{cgkt2m}
Let $(X, d)$ be a complete metric space with $|X| \geqslant 3$ and let the mapping
$T\colon X \to X$ be a continuous asymptotically regular generalized Kannan type mapping with the coefficient $\lambda \in [0,\infty)$. Then $T$ has a fixed point. The number of fixed points is at most two.
\end{cor}

\begin{rem}
 It is pertinent to mention that the continuity requirement for the mapping $T$ in Corollary~\ref{cgkt2m}  can be further relaxed. Corollary~\ref{cgkt2m} holds when we replace the continuity of the mapping $T$ with orbital continuity or $x_0$-orbitally continuity or almost orbitally continuity, or weakly orbitally continuity or $T$-orbitally lower semi-continuity or $k$-continuity. For more information on weaker continuity notions, see \cite{B23}.
\end{rem}

\begin{ex}
Let $X = [0, 1]$ and let $d$ be the Euclidean distance on $X$. Consider a mapping $T\colon X \to X$ such that  $T(x) = x/a$ for some real $a > 1$.
Example~\ref{exa3} it was shown that inequality~(\ref{e1}) holds for all $x > y > z$ iff $\lambda\geqslant 2/(a-1)$. Hence, if $a \in (3, \infty) $ ($a \in (1,\infty)$),  then $T$ is a generalized Kannan type mapping with some coefficient $\lambda \in (0,1)$ ($\lambda \in (0,\infty)$). Thus, contractions $T(x) = \alpha x$ with $\alpha \in (0, \frac{1}{3}) $ ($\alpha \in (0,1)$) are generalized Kannan type mapping with some coefficient $\lambda \in (0,1)$ ($\lambda \in (0,\infty)$).
\end{ex}

The following proposition is almost evident.
\begin{prop}\label{prop1}
Let $X$ be a finite nonempty metric space and let $T\colon X\to X$ be a self-mapping. Then $T$ is asymptotically regular if and only if $T$ does not possess periodic points of prime period $n\geqslant 2$.
\end{prop}

Proposition~\ref{prop1} and Corollary~\ref{cgkt2m} give the following.

\begin{cor}
Let $X$ be a finite nonempty metric space and let $T\colon X\to X$ be a self-mapping. If $T$ does not possess periodic points of prime period $n\geqslant 2$, then $T$ has a fixed point.
\end{cor}
\begin{proof}
It clear that $X$ is complete, $T$ is continuous and by Proposition~\ref{prop1} asymptotically regular. Suppose that $T$ has no fixed points. Hence, $d(x,Tx)+d(y,Ty)+d(z,Tz)\neq 0$ for all $x,y,z \in X$. It is easy to see that $T$ is a generalized Kannan type mapping with the coefficient
$$
\lambda =  \max_{x,y,z \in X}\frac{d(Tx,Ty)+d(Ty,Tz)+d(Tx,Tz)}{d(x,Tx)+d(y,Ty)+d(z,Tz)},
$$
where maximum is taken over all  pairwise distinct triplets of points $x,y,z \in X$.
Applying Corollary~\ref{cgkt2m} we get a contradiction.
\end{proof}

Note that a fixed point is a periodic point with the period $n=1$. Hence, we immediately obtain the following.

\begin{cor}
Let $X$ be a finite nonempty metric space and let $T\colon X\to X$ be a self-mapping. Then $T$  has a periodic point.
\end{cor}
This assertion seems already known and has also an elementary proof.
\begin{proof}
Let $x_0\in X$. Consider the sequence $x_n=T(x_{n-1})$, $n\geqslant 1$. By the pigeonhole principle, there exists at least one $n$, $1\leqslant n \leqslant |X|$, such that $x_n$ is equal to some earlier term in this sequence. Let $k$ be the smallest such index and let $i<k$ be such that $x_i=x_k$. Then $x_i$ is a periodic point with the prime period $k-i$.
\end{proof}

\section{Approximate fixed-point sequence and generalized $F$-Kannan type mapping}
Let $(X, d)$ be a metric space and $T\colon X \to X$.  Then a sequence $\{x_n\} \subset X$ is said to be an \emph{approximate fixed point sequence} of $T$ if  $d(x_n, Tx_n) \to 0$ as $n \to \infty$.

In the following example, we observe that while $T$ does not satisfy the asymptotic regularity condition, it does possess an approximate fixed-point sequence.

\begin{ex}
Let $X = [0, 1]$ be endowed with the usual metric, and $T : X \to X$ be defined by $T(x) = 1 - x$ for all $x \in [0, 1]$. It is important to note that $T$ is not asymptotically regular since every point of $X$ is a periodic point of prime period two, except the point $x=\frac{1}{2}$, which is the fixed point of $T$. However, the sequence $x_n=\frac{1}{2}+\frac{1}{n}$, $n\geqslant 2$ is the approximate fixed-point sequence of $T$.
Indeed,
$$
d(x_n,Tx_n)=\left|\frac{1}{2}+\frac{1}{n}-\left(1-\left(\frac{1}{2}+\frac{1}{n}\right)\right)\right|
=\frac{2}{n}\to 0
$$
as $n \to \infty$.
\end{ex}

In the following discussion, we examine mappings that may not exhibit asymptotic regularity but nonetheless fulfill the condition of generalized $F$-Kannan type mapping.

\begin{thm} \label{gkt5m}
Let $(X, d)$ be a complete metric space with $|X| \geqslant 3$ and let the mapping
$T\colon X \to X$ be a continuous, generalized $F$-Kannan type mapping. Suppose that $T$ has an approximate fixed-point sequence, i.e., there exists a sequence $\{x_n\} \subset X$ such that $d(x_n, Tx_n) \to 0$ as $n \to \infty$. Then, $T$ has a  fixed point. The number of fixed points is at most two.
\end{thm}

\begin{proof}

Using the triangle inequality and the definition of generalized $F$-Kannan type mapping, we obtain
\begin{align*}
    d(x_n, x_{n+p}) &\leq d(x_n, Tx_n) + d(Tx_n, Tx_{n+p}) + d(T x_{n+p}, x_{n+p}) \\
    &\leq d(x_n, T x_n) + d(Tx_n, Tx_{n+p}) + d(Tx_{n+p}, Tx_{n+p-1}) \\
    &+ d(Tx_{n+p-1}, Tx_{n})+ d(Tx_{n+p}, x_{n+p})\\
    & \leq d(x_n, Tx_n) + F(d(x_{n},Tx_n),d(x_{n+p},Tx_{n+p}), d(x_{n+p-1},Tx_{n+p-1})) \\
    & +d(Tx_{n+p}, x_{n+p}),
\end{align*}
which implies $d(x_n, x_{n+p}) \to 0$ as $n\to \infty$. Thus, $(x_{n})$ is a Cauchy sequence. Rest of the proof follows easily.
\end{proof}

\section{Fixed-point theorems in incomplete metric spaces}\label{sec5}

The following theorem is an analogue of Theorem 1 from~\cite{Ka69}. Note that in this theorem we omit the completeness of the metric space and have two new conditions (iii) and (iv).

\begin{thm}\label{t2}
Let $(X,d)$, $|X|\geqslant 3$, be a metric space and let the mapping $T\colon X\to X$ satisfy the following four conditions:
\begin{itemize}
  \item [(i)] $T$ does not possess periodic points of prime period $2$.
  \item [(ii)] $T$ is a generalized Kannan type mapping on $X$.
  \item [(iii)] $T$ is continuous at $x^*\in X$.
  \item [(iv)] There exists a point $x_0\in X$ such that the sequence of iterates $x_n=Tx_{n-1}$, $n=1,2,...$, has a subsequence $x_{n_k}$, converging to $x^*$.
\end{itemize}
Then $x^*$ is a fixed point of $T$. The number of fixed points is at most two.
\end{thm}
\begin{proof}
Since $T$ is continuous at $x^*$  and $x_{n_k} \to x^*$ we have $Tx_{n_k}=x_{n_{k}+1} \to Tx^*$. Note that $x_{n_{k}+1}$ is a subsequence of $x_n$ but not obligatory the subsequence of $x_{n_k}$. Suppose $x^*\neq Tx^*$. Consider two balls $B_1=B_1(x^*,r)$ and $B_2=B_2(Tx^*,r)$, where $r<\frac{1}{3}d(x^*, Tx^*)$. Consequently, there exists a positive integer $N$ such that $i>N$ implies
$$
x_{n_i} \in B_1 \, \text{ and } \, x_{n_i+1} \in B_2.
$$
Hence,
\begin{equation}\label{w4}
d(x_{n_i}, x_{n_i+1})>r
\end{equation}
for $i>N$.

If the sequence $x_n$ does not contain a fixed point of the mapping $T$, then we can apply considerations of Theorem~\ref{t1}.
By~(\ref{e9}) for $n=3,4,\ldots$ we have
$$
d(x_{n-1},x_n)\leqslant \alpha^{\frac{n}{2}-1}a,
$$
where $a=\max\{d(x_{0},x_{1}),d(x_{1},x_{2})\}$ and $\alpha=2\lambda/(2-\lambda)\in [0,1)$.
Hence,
$$
d(x_{n_i},x_{n_i+1})\leqslant \alpha^{\frac{n_i+1}{2}-1}a.
$$
But the last expression approaches $0$ as $i\to \infty$ which contradicts to~(\ref{w4}). Hence, $Tx^*=x^*$.

The existence of at most two fixed points follows form the last paragraph of Theorem~\ref{t1}.
\end{proof}

In the following theorem we suppose that $T$ is a generalized Kannan type mapping not on $X$ but on everywhere dense subset of $X$ and suppose that $f$ is continuous on $X$ but not only at the point $x^*$, compare with Theorem 2 from~\cite{Ka69}.

\begin{thm}\label{t3}
Let $(X,d)$, $|X|\geqslant 3$, be a metric space and let the mapping $T\colon X\to X$ be continuous. Suppose that
\begin{itemize}
  \item [(i)] $T$ does not possess periodic points of prime period $2$.
  \item [(ii)] $T$ is a generalized Kannan type mapping on $(M,d)$, where $M$ is an everywhere dense subset of $X$.
  \item [(iii)] There exists a point $x_0\in X$ such that the sequence of iterates $x_n=Tx_{n-1}$, $n=1,2,...$, has a subsequence $x_{n_k}$, converging to $x^*$.
\end{itemize}
Then $x^*$ is a fixed point of $T$. The number of fixed points is at most two.
\end{thm}
\begin{proof}
The proof will follow from Theorem~\ref{t2}, if we can show that $T$ is a generalized Kannan type mapping on $X$. Let $x, y, z$ be any three pairwise distinct points of $X$ such that  $x, y \in M$, $z\in X \setminus M$ and  let $(c_n)$ be a sequence in $M$ such that $c_n\to z$, $c_n\neq x$, $c_n\neq y$ for all $n$ and $c_i\neq c_j$, $i\neq j$. Then
$$
 d(Tx,Ty)+d(Ty,Tz)+d(Tx,Tz) \leqslant
$$
$$
d(Tx,Ty)+d(Ty,Tc_n)+d(Tc_n,Tz)+d(Tx,Tc_n)+d(Tc_n,Tz)
$$
$$
\leqslant \lambda(d(x,Tx)+d(y,Ty)+d(c_n,Tc_n))+2d(Tc_n,Tz)
$$
(using the inequality
\begin{equation}\label{s1}
  d(c_n,Tc_n)\leqslant d(c_n,z)+d(z,Tz)+d(Tz,Tc_n),
\end{equation}
we get)
$$
\leqslant \lambda(d(x,Tx)+d(y,Ty)+d(z,Tz))+\lambda d(c_n,z)+\lambda d(Tz,Tc_n)+2d(Tc_n,Tz).
$$
Letting $n\to \infty $ we get $d(c_n,z)\to 0$ and $d(Tc_n,Tz)\to 0$. Hence, inequality~(\ref{e1}) follows.

Let now  $x \in M$, $y, z\in X \setminus M$, and let $(b_n), (c_n)$ be sequences in $M$ such that $b_n\to y$ and $c_n\to z$. (Here and below we consider that the points $x, y, z$ and all elements of sequences converging to these points are pairwise distinct.) Then
$$
 d(Tx,Ty)+d(Ty,Tz)+d(Tx,Tz) \leqslant d(Tx,Tb_n)+d(Tb_n,Ty)
$$
$$
+d(Ty,Tb_n)+d(Tb_n,Tc_n)+d(Tc_n,Tz)+
d(Tx,Tc_n)+d(Tc_n,Tz)
$$
$$
\leqslant \lambda(d(x,Tx)+d(b_n,Tb_n)+d(c_n,Tc_n))+2d(Tb_n,Ty)
+2d(Tc_n,Tz)\leqslant
$$
(using the inequality
\begin{equation}\label{s2}
  d(b_n,Tb_n)\leqslant d(b_n,y)+d(y,Ty)+d(Ty,Tb_n)
\end{equation}
and inequality~(\ref{s1}) we get)
$$
\leqslant \lambda(d(x,Tx)+d(y,Ty)+d(z,Tz))
+2d(Tb_n,Ty)+2d(Tc_n,Tz)
$$
$$
+\lambda(d(b_n,y)+d(Ty,Tb_n)+d(c_n,z)+d(Tz,Tc_n)).
$$
Again, letting $n\to \infty $, we get inequality~(\ref{e1}).

Let now $x, y, z\in X \setminus M$, and let $(a_n), (b_n)$ and $(c_n)$ be sequences in $M$ such that $a_n\to x$, $b_n\to y$ and $c_n\to z$.
Then
$$
d(Tx,Ty)+d(Ty,Tz)+d(Tx,Tz)
$$
$$
\leqslant d(Tx,Ta_n)+d(Ta_n,Tb_n)+d(Tb_n,Ty)
$$
$$
+d(Ty,Tb_n)+d(Tb_n,Tc_n)+d(Tc_n,Tz)
$$
$$
+d(Tx,Ta_n)+d(Ta_n,Tc_n)+d(Tc_n,Tz)
$$
$$
\leqslant \lambda(d(a_n,Ta_n)+d(b_n,Tb_n)+d(c_n,Tc_n))
$$
$$
+2d(Ta_n,Tx)
+2d(Tb_n,Ty)
+2d(Tc_n,Tz)
$$
(using the inequality
\begin{equation*}
  d(a_n,Ta_n)\leqslant d(a_n,x)+d(x,Tx)+d(Tx,Ta_n)
\end{equation*}
and inequalities~(\ref{s1}) and~(\ref{s2}) we get)
$$
\leqslant \lambda(d(x,Tx)+d(y,Ty)+d(z,Tz))
$$
$$
+2d(Ta_n,Tx)
+2d(Tb_n,Ty)
+2d(Tc_n,Tz)
$$
$$
+\lambda(
d(a_n,x)+d(Tx,Ta_n)+
d(b_n,y)+d(Ty,Tb_n)+d(c_n,z)+d(Tz,Tc_n)).
$$
Again, letting $n\to \infty$, we get inequality~(\ref{e1}). Hence, $T$ is a generalized Kannan type mapping on $X$, which completes the proof.
\end{proof}

\noindent\textbf{Acknowledgements}

\noindent This work was partially supported by a grant from the Simons Foundation (Award 1160640, Presidential Discretionary-Ukraine Support Grants, E. Petrov)


\begin{thebibliography}{31}
\ifx \bisbn   \undefined \def \bisbn  #1{ISBN #1}\fi
\ifx \binits  \undefined \def \binits#1{#1}\fi
\ifx \bauthor  \undefined \def \bauthor#1{#1}\fi
\ifx \batitle  \undefined \def \batitle#1{#1}\fi
\ifx \bjtitle  \undefined \def \bjtitle#1{#1}\fi
\ifx \bvolume  \undefined \def \bvolume#1{\textbf{#1}}\fi
\ifx \byear  \undefined \def \byear#1{#1}\fi
\ifx \bissue  \undefined \def \bissue#1{#1}\fi
\ifx \bfpage  \undefined \def \bfpage#1{#1}\fi
\ifx \blpage  \undefined \def \blpage #1{#1}\fi
\ifx \burl  \undefined \def \burl#1{\textsf{#1}}\fi
\ifx \doiurl  \undefined \def \doiurl#1{\url{https://doi.org/#1}}\fi
\ifx \betal  \undefined \def \betal{\textit{et al.}}\fi
\ifx \binstitute  \undefined \def \binstitute#1{#1}\fi
\ifx \binstitutionaled  \undefined \def \binstitutionaled#1{#1}\fi
\ifx \bctitle  \undefined \def \bctitle#1{#1}\fi
\ifx \beditor  \undefined \def \beditor#1{#1}\fi
\ifx \bpublisher  \undefined \def \bpublisher#1{#1}\fi
\ifx \bbtitle  \undefined \def \bbtitle#1{#1}\fi
\ifx \bedition  \undefined \def \bedition#1{#1}\fi
\ifx \bseriesno  \undefined \def \bseriesno#1{#1}\fi
\ifx \blocation  \undefined \def \blocation#1{#1}\fi
\ifx \bsertitle  \undefined \def \bsertitle#1{#1}\fi
\ifx \bsnm \undefined \def \bsnm#1{#1}\fi
\ifx \bsuffix \undefined \def \bsuffix#1{#1}\fi
\ifx \bparticle \undefined \def \bparticle#1{#1}\fi
\ifx \barticle \undefined \def \barticle#1{#1}\fi
\bibcommenthead
\ifx \bconfdate \undefined \def \bconfdate #1{#1}\fi
\ifx \botherref \undefined \def \botherref #1{#1}\fi
\ifx \url \undefined \def \url#1{\textsf{#1}}\fi
\ifx \bchapter \undefined \def \bchapter#1{#1}\fi
\ifx \bbook \undefined \def \bbook#1{#1}\fi
\ifx \bcomment \undefined \def \bcomment#1{#1}\fi
\ifx \oauthor \undefined \def \oauthor#1{#1}\fi
\ifx \citeauthoryear \undefined \def \citeauthoryear#1{#1}\fi
\ifx \endbibitem  \undefined \def \endbibitem {}\fi
\ifx \bconflocation  \undefined \def \bconflocation#1{#1}\fi
\ifx \arxivurl  \undefined \def \arxivurl#1{\textsf{#1}}\fi
\csname PreBibitemsHook\endcsname

\bibitem[\protect\citeauthoryear{Kannan}{1968}]{Ka68}
\begin{barticle}
\bauthor{\bsnm{Kannan}, \binits{R.}}:
\batitle{Some results on fixed points}.
\bjtitle{Bull. Calcutta Math. Soc.}
\bvolume{60},
\bfpage{71}--\blpage{76}
(\byear{1968})
\end{barticle}
\endbibitem

\bibitem[\protect\citeauthoryear{Subrahmanyam}{1975}]{S75}
\begin{barticle}
\bauthor{\bsnm{Subrahmanyam}, \binits{P.V.}}:
\batitle{Completeness and fixed-points}.
\bjtitle{Monatsh. Math.}
\bvolume{80}(\bissue{4}),
\bfpage{325}--\blpage{330}
(\byear{1975})
\end{barticle}
\endbibitem

\bibitem[\protect\citeauthoryear{Connell}{1959}]{Co59}
\begin{barticle}
\bauthor{\bsnm{Connell}, \binits{E.H.}}:
\batitle{Properties of fixed point spaces}.
\bjtitle{Proc. Amer. Math. Soc.}
\bvolume{10},
\bfpage{974}--\blpage{979}
(\byear{1959})
\end{barticle}
\endbibitem

\bibitem[\protect\citeauthoryear{Shioji et~al.}{1998}]{SST98}
\begin{barticle}
\bauthor{\bsnm{Shioji}, \binits{N.}},
\bauthor{\bsnm{Suzuki}, \binits{T.}},
\bauthor{\bsnm{Takahashi}, \binits{W.}}:
\batitle{Contractive mappings, {K}annan mappings and metric completeness}.
\bjtitle{Proc. Amer. Math. Soc.}
\bvolume{126}(\bissue{10}),
\bfpage{3117}--\blpage{3124}
(\byear{1998})
\end{barticle}
\endbibitem

\bibitem[\protect\citeauthoryear{Suzuki}{2005}]{Su05}
\begin{barticle}
\bauthor{\bsnm{Suzuki}, \binits{T.}}:
\batitle{Contractive mappings are {K}annan mappings, and {K}annan mappings are
  contractive mappings in some sense}.
\bjtitle{Comment. Math. (Prace Mat.)}
\bvolume{45}(\bissue{1}),
\bfpage{45}--\blpage{58}
(\byear{2005})
\end{barticle}
\endbibitem

\bibitem[\protect\citeauthoryear{Kikkawa and Suzuki}{2008a}]{KT08}
\begin{botherref}
\oauthor{\bsnm{Kikkawa}, \binits{M.}},
\oauthor{\bsnm{Suzuki}, \binits{T.}}:
Some similarity between contractions and {K}annan mappings.
Fixed Point Theory Appl.
(2008).
Art. ID 649749, 8 pages
\end{botherref}
\endbibitem

\bibitem[\protect\citeauthoryear{Kikkawa and Suzuki}{2008b}]{KS08}
\begin{botherref}
\oauthor{\bsnm{Kikkawa}, \binits{M.}},
\oauthor{\bsnm{Suzuki}, \binits{T.}}:
Some similarity between contractions and {K}annan mappings. {II}.
Bull. Kyushu Inst. Technol. Pure Appl. Math.
(55),
1--13
(2008)
\end{botherref}
\endbibitem

\bibitem[\protect\citeauthoryear{De~la Sen}{2009}]{DLS09}
\begin{botherref}
\oauthor{\bsnm{Sen}, \binits{M.}}:
Some combined relations between contractive mappings, {K}annan mappings,
  reasonable expansive mappings, and {$T$}-stability.
Fixed Point Theory Appl.
(2009).
Art. ID 815637, 25 pages
\end{botherref}
\endbibitem

\bibitem[\protect\citeauthoryear{De~la Sen}{2010}]{DLS10}
\begin{botherref}
\oauthor{\bsnm{Sen}, \binits{M.}}:
Linking contractive self-mappings and cyclic {M}eir-{K}eeler contractions with
  {K}annan self-mappings.
Fixed Point Theory Appl.
(2010).
Art. ID 572057, 23 pages
\end{botherref}
\endbibitem

\bibitem[\protect\citeauthoryear{Berinde}{2004}]{Be04}
\begin{barticle}
\bauthor{\bsnm{Berinde}, \binits{V.}}:
\batitle{Approximating fixed points of weak contractions using the {P}icard
  iteration}.
\bjtitle{Nonlinear Anal. Forum}
\bvolume{9}(\bissue{1}),
\bfpage{43}--\blpage{53}
(\byear{2004})
\end{barticle}
\endbibitem

\bibitem[\protect\citeauthoryear{Berinde}{2007}]{Be07}
\begin{bbook}
\bauthor{\bsnm{Berinde}, \binits{V.}}:
\bbtitle{Iterative Approximation of Fixed Points}.
\bpublisher{Springer},
\blocation{Berlin}
(\byear{2007})
\end{bbook}
\endbibitem

\bibitem[\protect\citeauthoryear{Is\'{e}ki}{1974}]{Is74}
\begin{barticle}
\bauthor{\bsnm{Is\'{e}ki}, \binits{K.}}:
\batitle{Generalizations of {K}annan fixed point theorems}.
\bjtitle{Math. Sem. Notes Kobe Univ.}
\bvolume{2}(\bissue{1}),
\bfpage{1}--\blpage{5}
(\byear{1974}).
\bcomment{(Paper No.~5)}
\end{barticle}
\endbibitem

\bibitem[\protect\citeauthoryear{Reich}{1971a}]{Re71}
\begin{barticle}
\bauthor{\bsnm{Reich}, \binits{S.}}:
\batitle{Some remarks concerning contraction mappings}.
\bjtitle{Canad. Math. Bull.}
\bvolume{14},
\bfpage{121}--\blpage{124}
(\byear{1971})
\end{barticle}
\endbibitem

\bibitem[\protect\citeauthoryear{Reich}{1971b}]{Re712}
\begin{barticle}
\bauthor{\bsnm{Reich}, \binits{S.}}:
\batitle{Kannan's fixed point theorem}.
\bjtitle{Boll. Un. Mat. Ital. (4)}
\bvolume{4},
\bfpage{1}--\blpage{11}
(\byear{1971})
\end{barticle}
\endbibitem

\bibitem[\protect\citeauthoryear{Rhoades}{1977}]{Rh77}
\begin{barticle}
\bauthor{\bsnm{Rhoades}, \binits{B.E.}}:
\batitle{A comparison of various definitions of contractive mappings}.
\bjtitle{Trans. Amer. Math. Soc.}
\bvolume{226},
\bfpage{257}--\blpage{290}
(\byear{1977})
\end{barticle}
\endbibitem

\bibitem[\protect\citeauthoryear{Bianchini}{1972}]{Bi72}
\begin{barticle}
\bauthor{\bsnm{Bianchini}, \binits{R.}}:
\batitle{Su un problema di {S}. {R}eich riguardante la teoria dei punti fissi}.
\bjtitle{Boll. Un. Mat. Ital. (4)}
\bvolume{5},
\bfpage{103}--\blpage{108}
(\byear{1972})
\end{barticle}
\endbibitem

\bibitem[\protect\citeauthoryear{G\'ornicki}{2019}]{Go19}
\begin{barticle}
\bauthor{\bsnm{G\'ornicki}, \binits{J.}}:
\batitle{Remarks on asymptotic regularity and fixed points}.
\bjtitle{J. Fixed Point Theory Appl.}
\bvolume{21}(\bissue{1}),
\bfpage{1}--\blpage{20}
(\byear{2019}).
\bcomment{(Paper No. 29)}
\end{barticle}
\endbibitem

\bibitem[\protect\citeauthoryear{Akkouchi}{2022}]{Ak22}
\begin{barticle}
\bauthor{\bsnm{Akkouchi}, \binits{M.}}:
\batitle{On a family of {K}annan type selfmaps of {$b$}-metric spaces}.
\bjtitle{Appl. Math. E-Notes}
\bvolume{22},
\bfpage{751}--\blpage{765}
(\byear{2022})
\end{barticle}
\endbibitem

\bibitem[\protect\citeauthoryear{Karahan and Isik}{2019}]{KI19}
\begin{barticle}
\bauthor{\bsnm{Karahan}, \binits{I.}},
\bauthor{\bsnm{Isik}, \binits{I.}}:
\batitle{Generalization of {B}anach, {K}annan and {C}iric fixed point theorems
  in {$b_v(s)$} metric spaces}.
\bjtitle{Politehn. Univ. Bucharest Sci. Bull. Ser. A Appl. Math. Phys.}
\bvolume{81}(\bissue{1}),
\bfpage{73}--\blpage{80}
(\byear{2019})
\end{barticle}
\endbibitem

\bibitem[\protect\citeauthoryear{Asgari and Abbasbigi}{2012}]{AA12}
\begin{barticle}
\bauthor{\bsnm{Asgari}, \binits{M.S.}},
\bauthor{\bsnm{Abbasbigi}, \binits{Z.}}:
\batitle{Generalizations of the {S}uzuki and {K}annan fixed point theorems in
  {$G$}-cone metric spaces}.
\bjtitle{Anal. Theory Appl.}
\bvolume{28}(\bissue{3}),
\bfpage{248}--\blpage{262}
(\byear{2012})
\end{barticle}
\endbibitem

\bibitem[\protect\citeauthoryear{Dominguez et~al.}{2012}]{DLG12}
\begin{barticle}
\bauthor{\bsnm{Dominguez}, \binits{T.}},
\bauthor{\bsnm{Lorenzo}, \binits{J.}},
\bauthor{\bsnm{Gatica}, \binits{I.}}:
\batitle{Some generalizations of {K}annan's fixed point theorem in {$K$}-metric
  spaces}.
\bjtitle{Fixed Point Theory}
\bvolume{13}(\bissue{1}),
\bfpage{73}--\blpage{83}
(\byear{2012})
\end{barticle}
\endbibitem

\bibitem[\protect\citeauthoryear{Homorodan}{2022}]{Ho22}
\begin{botherref}
\oauthor{\bsnm{Homorodan}, \binits{P.}}:
Fixed point theorems for discontinuous mappings of {K}annan and {B}ianchini
  type in distance spaces.
Thai J. Math.
\textbf{20}(3)
(2022)
\end{botherref}
\endbibitem

\bibitem[\protect\citeauthoryear{Nakanishi and Suzuki}{2010}]{NS10}
\begin{bchapter}
\bauthor{\bsnm{Nakanishi}, \binits{M.}},
\bauthor{\bsnm{Suzuki}, \binits{T.}}:
\bctitle{A fixed point theorem for some {K}annan-type set-valued mappings}.
In: \bbtitle{Nonlinear Analysis and Convex Analysis},
pp. \bfpage{249}--\blpage{255}.
\bpublisher{Yokohama Publ.},
\blocation{Yokohama}
(\byear{2010})
\end{bchapter}
\endbibitem

\bibitem[\protect\citeauthoryear{Ume}{2015}]{Um15}
\begin{botherref}
\oauthor{\bsnm{Ume}, \binits{J.S.}}:
Fixed point theorems for {K}annan-type maps.
Fixed Point Theory Appl.,
2015--3813
(2015)
\end{botherref}
\endbibitem

\bibitem[\protect\citeauthoryear{Damjanovi\'{c} and Dori\'{c}}{2011}]{DD11}
\begin{barticle}
\bauthor{\bsnm{Damjanovi\'{c}}, \binits{B.}},
\bauthor{\bsnm{Dori\'{c}}, \binits{D.}}:
\batitle{Multivalued generalizations of the {K}annan fixed point theorem}.
\bjtitle{Filomat}
\bvolume{25}(\bissue{1}),
\bfpage{125}--\blpage{131}
(\byear{2011})
\end{barticle}
\endbibitem

\bibitem[\protect\citeauthoryear{Petrov}{2023}]{P23}
\begin{barticle}
\bauthor{\bsnm{Petrov}, \binits{E.}}:
\batitle{Fixed point theorem for mappings contracting perimeters of triangles}.
\bjtitle{J. Fixed Point Theory Appl.}
\bvolume{25},
\bfpage{1}--\blpage{11}
(\byear{2023}).
\bcomment{(paper no. 74)}
\end{barticle}
\endbibitem

\bibitem[\protect\citeauthoryear{Kannan}{1969}]{Ka69}
\begin{barticle}
\bauthor{\bsnm{Kannan}, \binits{R.}}:
\batitle{Some results on fixed points. {II}}.
\bjtitle{Amer. Math. Monthly}
\bvolume{76},
\bfpage{405}--\blpage{408}
(\byear{1969})
\end{barticle}
\endbibitem

\bibitem[\protect\citeauthoryear{Devaney}{2003}]{De22}
\begin{bbook}
\bauthor{\bsnm{Devaney}, \binits{R.L.}}:
\bbtitle{An Introduction to Chaotic Dynamical Systems}.
\bpublisher{Westview Press},
\blocation{Boulder, CO}
(\byear{2003})
\end{bbook}
\endbibitem

\bibitem[\protect\citeauthoryear{Rhoades}{1988}]{Rh88}
\begin{barticle}
\bauthor{\bsnm{Rhoades}, \binits{B.E.}}:
\batitle{Contractive definitions and continuity}.
\bjtitle{Contemporary Mathematics}
\bvolume{72},
\bfpage{233}--\blpage{245}
(\byear{1988})
\end{barticle}
\endbibitem

\bibitem[\protect\citeauthoryear{Browder and Petryshyn}{1966}]{BP66}
\begin{barticle}
\bauthor{\bsnm{Browder}, \binits{F.E.}},
\bauthor{\bsnm{Petryshyn}, \binits{W.V.}}:
\batitle{The solution by iteration of nonlinear functional equations in
  {B}anach spaces}.
\bjtitle{Bull. Amer. Math. Soc.}
\bvolume{72},
\bfpage{571}--\blpage{575}
(\byear{1966})
\end{barticle}
\endbibitem

\bibitem[\protect\citeauthoryear{Bisht}{2023}]{B23}
\begin{barticle}
\bauthor{\bsnm{Bisht}, \binits{R.K.}}:
\batitle{An overview of the emergence of weaker continuity notions, various
  classes of contractive mappings and related fixed point theorems}.
\bjtitle{J. Fixed Point Theory Appl.}
\bvolume{25},
\bfpage{1}--\blpage{29}
(\byear{2023}).
\bcomment{(Paper No. 11)}
\end{barticle}
\endbibitem

\end{thebibliography}
\end{document}